\documentclass[11pt,british]{amsart}

\usepackage{babel,eucal,url,amssymb,enumerate,amscd,graphics}

\theoremstyle{plain}
\newtheorem{lemma}{Lemma}[section]
\newtheorem{definition}[lemma]{Definition}
\newtheorem{proposition}[lemma]{Proposition}
\newtheorem{corollary}[lemma]{Corollary}
\newtheorem{theorem}[lemma]{Theorem}
\newtheorem{remark}[lemma]{Remark}
\newtheorem{example}[lemma]{Example}

\newcommand{\Lie}[1]{\operatorname{\textsl{#1}}}

\newcommand{\Gtwo}{\ifmmode{{\rm G}_2}\else{${\rm G}_2$}\fi}

\def\sideremark#1{\ifvmode\leavevmode\fi\vadjust{\vbox to0pt{\vss
 \hbox to 0pt{\hskip\hsize\hskip1em
 \vbox{\hsize2.5cm\tiny\raggedright\pretolerance10000
 \noindent #1\hfill}\hss}\vbox to8pt{\vfil}\vss}}}%

\newfont{\eusm}{eusm10 scaled \magstep1}
\newfont{\eusmiii}{eusm10 scaled \magstep3}
\newcommand{\got}[1]{\mbox{\eusm #1}}

\newcommand{\comp}{\makebox[7pt]{\raisebox{1.5pt}{\tiny $\circ$}}}
\newcommand{\RR}{\mbox{{\sl I}}\!\mbox{{\sl R}}}

\newcommand{\trace}{\mathop{\rm trace}}

\title{Energy of generalized distributions}

\author{J.~C.~Gonz{\'a}lez-D{\'a}vila}
\address[J.~C.~Gonz{\'a}lez-D{\'a}vila]{Departamento de Matem\'aticas, Estad\'istica e Investigaci\'on Ope\-ra\-tiva,
  University of La Laguna\\ 38200 La Laguna, Tenerife, Spain}
\email{jcgonza@ull.es}

\date{\today}

\begin{document}

\maketitle

\begin{abstract}{\indent}
We consider the energy of smooth generalized distributions and also of singular foliations on compact Riemannian manifolds for which the set of their singularities consists of a finite number of isolated points and of pairwise disjoint closed submanifolds. We derive a lower bound for the energy of all $q$-dimensional almost regular distributions, for each $q < \dim M,$ and find several examples of foliations which minimize the energy functional over certain sets of smooth generalized distributions.

\vspace{4mm}

 \vspace{2mm}

\end{abstract}
\begin{figure}[b]  \vspace{-5mm}

\hspace{-5.7cm}
{\footnotesize Supported by D.G.I. (Spain) Project MTM2013-46961-P.}
\end{figure}

\section{Introduction}\indent
Let $\sigma:x\in M\mapsto \sigma(x) \subset T_{x}M$ be a smooth {\em generalized} {\em distribution} \cite{Suss} on an $n$-dimensional compact and connected Riemannian manifold $(M,g).$ Denote by $M_{r}$  the subset of its regular points. When the restriction $\sigma_{r}$ of $\sigma$ to $M_{r}$ is a regular distribution, that is, the lower semicontinuous function $d,$ given by $d(x) = \dim \sigma(x),$ is a constant $q$ on $M_{r},$ $\sigma$ is said to be an $q$-dimensional {\em almost regular distribution}. If $d$ is constant on whole $M,$ $\sigma$ is the classical distribution. For the sake of brevity, we will refer to generalized distributions simply as {\em distributions} and we will use the term {\em regular} for the classical distribution.

In the general case, we can only guarantee that $d$ is constant on each one of the connected components of $M_{r}.$ Let $1\leq q_{1}<q_{2}<\dots <q_{l}\leq n= \dim M$ be the values of $d$ on $M_{r}$ and put $M_{r}^{i} :=\{x\in M_{r}\mid d(x) = q_{i}\},$ $i\in \{1,\dots ,l\},$ the union of the connected components, supposed to be oriented, on which $d$ is constant equals to $q_{i}.$ Then $\sigma_{r}$ is the union of the regular $q^{i}$-dimensional distributions $\sigma^{i}$ on $M_{r}^{i},$ $i = 1,\dots ,l,$ and it could be seen as a smooth section of the Grassmannian bundle $\pi: G(M_{r}) = \bigcup_{i=1}^{l}G_{q_{i}}(M_{r}^{i})\to M_{r},$ where $G_{q_{i}}(M_{r}^{i}) = \bigcup_{x\in M_r^{i}}G_{q_{i}}(T_{x}M_{r}^{i})$ is the Grassmannian bundle of the $q_{i}$-dimensional linear subspaces in the tangent space $TM_{r}^{i}.$ Because $G_{q}(M_{r}^{i})$ is diffeomorphic to the {\em homogeneous fibre bundle} ${\mathcal S}{\mathcal O}(M_{r}^{i})/S(O(q_{i})\times O(n-q_{i})),$ ${\mathcal S}{\mathcal O}(M_{r}^{i})$ being the principal $SO(n)$-bundle of oriented orthonormal frames of $(M_{r}^{i},g_{M_{r}^{i}}),$ $G(M_{r})$ will be provided (see Section $3)$ with a natural Riemannian metric $g^{K},$ known as the {\em Kaluza-Klein metric} \cite{Wood}. It makes $\pi:(G(M_{r}),g^{K})\to (M_{r},g_{M_{r}})$ a Riemannian submersion with totally geodesic fibres, where $g_{M_{r}}$ denotes the induced metric by $g$ on $M_{r}.$

The energy of a $q$-dimensional regular distribution $\sigma$ is defined in \cite{GDHarm} (see also \cite{Choi-Yim} and \cite{Wood1}) as the energy of the map $\sigma: (M,g)\to (G_{q}(M),g^{K}).$ An equivalent definition for oriented regular distributions, considered as sections of the bundle of unit decomposable $q$-vectors equipped with the {\em generalized Sasaki metric}, is given in \cite{ChNW} and \cite{GGV}, among others. For an arbitrary map $\sigma:(M, g)\to (N, h)$ between Riemannian
manifolds, $M$ being compact and oriented, the {\em energy} of $\sigma$ is the integral
\begin{equation}\label{energyf}
 {\mathcal E}(\sigma)= \frac{\textstyle 1}{\textstyle2}\int_{M}\trace L_{\sigma}\;dv_{M},
\end{equation}
where $L_{\sigma}$ is the $(1,1)$-tensor field determined by $(\sigma^{*}h)(X,Y) = g(L_{\sigma}X,Y),$ for all vector fields $X,Y,$ and $dv_{M}$ denotes the volume form on $(M,g).$ (For more information about the energy functional see \cite{Ur}.) We are particularly interested on the energy functional of smooth distributions whose set of singular points $M_{s}= M\setminus M_{r}$ is given by the union
\begin{equation}\label{set}
M_{s} = \{x_{1},\dots ,x_{a}\}\cup \Big(\bigcup_{\beta=1}^{b}P_{\beta}\Big)
\end{equation}
of a finite number of points $x_{1},\dots ,x_{a}$ of $M$ and of pairwise disjoint topologically embedded submanifolds $P_{\beta},$ $\beta=1,\dots ,b,$ with $1\leq \dim P_{\beta}\leq n-1.$ In fact, the main purpose of this paper is not the study of variational problems of this functional, but of the energy itself, of smooth distributions and singular foliations, as a natural extension from the theory about the energy of unit vector fields, with singularities as in (\ref{set}), developed initially by Brito and Walczak in \cite{BW} and, later, by Boeckx, Vanhecke and the author in \cite{BGV}. Note that a unit vector field with singularities determines an {\em oriented} one-dimensional almost regular distribution. 

In terms of $G$-structures, each regular $q^{i}$-dimensional distribution $\sigma^{i},$ $i = 1,\dots ,l,$ co\-rres\-ponds with a (unique) $S(O(q_{i})\times O(n-q_{i}))$-reduction of ${\mathcal S}{\mathcal O}(M_{r}^{i}).$ Then, using the {\em intrinsic torsion} $\xi$ of each one of these $S(O(q_{i})\times O(n-q_{i}))$-structures, the energy ${\mathcal E}(\sigma)$ of $\sigma$ is expressed in Section $3$ as
\[
{\mathcal E}(\sigma) = \frac{n}{2}{\rm Vol}(M,g) + \frac{1}{4}\int_{M}\|\xi\|^{2}dv_{M}.
\]

If $\sigma$ is completely integrable, the connected components of the maximal integral sub\-ma\-ni\-folds are the leaves of a foliation ${\mathcal F} = {\mathcal F}_{\sigma}$ with singularities known as a {\em Stefan foliation} or a {\em singular foliation} \cite{Stefan}. The energy ${\mathcal E}({\mathcal F})$ of ${\mathcal F}$ is defined as the energy $ {\mathcal E}(\sigma)$ of the {\em tangent distribution} $\sigma$ of ${\mathcal F}.$

In Section $4,$ an useful integral formula for almost regular distributions, with finite ener\-gy and $M_{r}$ connected, of its {\em mixed scalar curvature} and of its {\em second mean curvatures} is obtained. This formula, which can be seen as a generalization the one given in \cite{BGV}, see also \cite{BW}, for unit vector fields with singularities, plays a central role for the determination of some lower bounds of the energy of smooth distributions. Thus, in Section $5$ we show that tori are the unique compact oriented surfaces admitting a one-dimensional almost regular distribution with finite energy, and for $n\geq 3,$ we derive a lower bound for the energy in the set of all $q$-dimensional almost regular distributions, for each $q = 1,\dots ,n-1.$ As an application of this last result, we find in Section $6$ some special classes of foliations, as tubular and radial foliations around points or embedded submanifolds, particularly on compact rank one symmetric spaces, and complex radial foliations in Hermitian geometry, that mi\-ni\-mi\-ze the energy functional at least over a sufficiently wide set of smooth distributions. We show (Theorem \ref{Thsph}) that radial and spherical foliations around points in the sphere and also in the real projective space are the {\em unique} absolute minimizers of the energy functional over the set of all one-dimensional and codimension one almost regular distributions, respectively. Moreover, we construct a family of foliations, obtained by deformation of a tubular fo\-lia\-tion, which are not almost regular and have finite energy even when they are on surfaces different from tori.
\section{Intrinsic torsion of smooth distributions}\indent
A {\em distribution} \cite{Suss} on a differentiable manifold $M$ is a mapping $\sigma$ which assigns to every $x\in M$ a linear subspace $\sigma(x)$ of the tangent space $T_{x}M.$ The subspaces $\sigma(x)$ may have different dimensions. Denote by ${\mathfrak X}_{loc}(M)$ the set of all $C^{\infty}$ vector fields defined on open subsets of $M.$ A vector field $X\in {\mathfrak X}_{loc}(M)$ is said that {\em belongs} to the distribution $\sigma,$ we shall put $X\in \sigma,$ if $X_{x}\in \sigma(x)$ for every $x$ in the domain of $X$ and a subset $D\subset {\mathfrak X}_{loc}(M)$ is said to {\em span} $\sigma$ if, for every $x\in M,$ $\sigma(x)$ is the linear hull of vectors $X_{x},$ where $X\in D.$ The distribution $\sigma$ is then called a {\em smooth} or $C^{\infty}$ {\em distribution}. In particular, any finite collection of vector fields determines a smooth distribution.

A point $x\in M$ will be a {\em regular point} if $x$ is a local maximum of $d$ or, equivalently, $d$ is constant on an open neighborhood of $x.$ All the other points of $M$ will be called {\em singular points} of $\sigma.$ Then $M = M_{r}\cup M_{s},$ where $M_{r}$ and $M_{s}$ denote the set of the regular and singular points of $\sigma,$ respectively. $M_{r}$ is obviously an open dense subset of $M.$

Let $\pi_{SO(n)}: {\mathcal S}{\mathcal O}(M_{r}) \to M_{r}$ be the principal $SO(n)$-bundle of oriented orthonormal frames of $(M_{r},g_{M_{r}}),$ consisting on the pairs $p=(x;p_{1},\dots ,p_{n})$ where $x\in M_{r}$ and $\{p_{1},\dots ,p_{n}\}$ is an oriented and orthonormal basis of $(T_{x}M_{r},g_{M_{r}}).$ Then, ${\mathcal S}{\mathcal O}(M_{r}) = \bigcup_{i=1}^{l}{\mathcal S}{\mathcal O}(M_{r}^{i})$ and $G_{q_{i}}(M_{r}^{i}),$ for each $i = 1,\dots ,l,$ can be identified with the orbit space ${\mathcal S}{\mathcal O}(M_{r}^{i})/S(O(q_{i})\times O(n-q_{i})),$ via the mapping $p\cdot S(O(q_{i})\times O(n-q_{i}))\mapsto \RR\{p_{1},\dots ,p_{q_{i}}\}.$ So $G_{q_{i}}(M_{r}^{i})$ is a {\em homogeneous fibre bundle} \cite{Wood} with fibre type the real Grassmannian manifold $G_{q_{i}}(\RR^{n}) = SO(n)/S(O(q_{i})\times O(n-q_{i}))$ of the unoriented $q_{i}$-subspaces of $\RR^{n}.$

Let $\rho: {\mathcal S}{\mathcal O}(M_{r})\to G(M_{r})$ be the map whose restriction to each ${\mathcal S}{\mathcal O}(M_{r}^{i})$ is the orbit map $p\mapsto p\cdot S(O(q_{i}\times O(n-q_{i})).$ Then $\pi_{{\mathcal S}{\mathcal O}(n)} = \pi\comp \rho$ and each section in $\Gamma^{\infty}(G_{q_{i}}M_{r}^{i})$ determines a reduction ${\mathcal S}{\mathcal O}_{\sigma^{i}}(M_{r}^{i})$ to $S(O(q_{i})\times O(n-q_{i}))$ of ${\mathcal S}{\mathcal O}(M_{r}^{i})$ and conversely. 

Hence, there is a one-to-one correspondence between the set of $S(O(q_{i})\times O(n-q_{i}))$-structures and the manifold $\Gamma^{\infty}(G_{q_{i}}(M_{r}^{i}))$ of all $q_{i}$-dimensional distributions of $(M_{r}^{i},g_{M_{r}^{i}}).$ Moreover, the pair $({\mathcal V}^{i}= \sigma^{i},{\mathcal H}^{i}=(\sigma^{i})^{\bot}),$ $i = 1,\dots ,l,$ where $(\sigma^{i})^{\bot}$ is the orthogonal distribution of $\sigma^{i}$ on $(M_{r}^{i},g_{M_{r}^{i}}),$ de\-ter\-mi\-nes a {\em Riemannian almost-product $(AP)$ structure}, i.e., an orthogonal $(1,1)$-tensor field $P^{i}$ on $(M_{r}^{i},g_{M_{r}^{i}})$ with $(P^{i})^{2} = {\rm Id}$ and $P^{i}\neq \pm {\rm Id}.$ The {\em vertical} and {\em horizontal} distributions ${\mathcal V}^{i}$ and ${\mathcal H}^{i}$ are the corresponding $\pm 1$-eigendistributions of $P^{i}.$

Denote by ${\mathfrak s}{\mathfrak o}(M_{r}^{i})_{\sigma^{i}}$ the subbundle of ${\mathfrak s}{\mathfrak o}(M^{i}_{r})$ of the skew-symmetric endomorphisms $A$ of $TM_{r}^{i}$ such that $AP^{i} = -P^{i}A.$ Following \cite{GDHarm} (see also \cite{GDM}), the {\em minimal connection} of $\sigma^{i},$ considered as a $S(O(q_{i})\times O(n-q_{i}))$-structure, is the unique $S(O(q_{i})\times O(n-q_{i}))$-connection $\nabla^{\sigma^{i}} = \nabla- \xi^{i}$ on $M_{r}^{i},$ where $\nabla$ is the Levi-Civita connection of $(M,g)$ and $\xi^{i},$ known as the {\em intrinsic torsion} of $\sigma^{i},$ is an element of $T^{*}M^{i}_{r}\otimes {\mathfrak s}{\mathfrak o}(M_{r}^{i})_{\sigma^{i}}.$ Then $\nabla^{\sigma^{i}}$ coincides with the {\em Schouten connection} of the $AP$ structure and $\xi^{i}$ is given by
\begin{equation}\label{intrixi}
(\xi^{i})_{X}Y = -\frac{1}{2}P^{i}(\nabla_{X}P^{i})Y,\;\;\;X,Y\in {\mathfrak X}(M_{r}^{i}).
\end{equation}

Let $({\mathcal V},{\mathcal H})$ be the complementary (smooth) distributions on $M_{r}$ obtained taking the pairs $({\mathcal V}^{i},{\mathcal H}^{i})$ on each $M_{r}^{i},$ for $i\in \{1,\dots ,l\},$ and let $p_{\mathcal V}= \frac{1}{2}({\rm Id} + P): TM_{r} \to {\mathcal V}$ and $p_{\mathcal H} = \frac{1}{2}({\rm Id} - P) : TM_{r} \to {\mathcal H}$ be the canonical projections, where $P$ denotes the $(1,1)$-tensor field on $M_{r}$ associated to $({\mathcal V},{\mathcal H}).$ The $(1,2)$-tensor field $\xi$ on $M_{r},$ defined by $\xi(x) = \xi^{i}(x)$ if $x\in M_{r}^{i},$ is called the {\em intrinsic torsion} of $\sigma.$ Then $\xi_{X},$ for all $X\in {\mathfrak X}(M_{r}),$ is a global section of the vector bundle
\[
{\mathfrak s}{\mathfrak o}(M_{r})_{\sigma_{r}}= \bigcup_{i=1}^{l} {\mathfrak s}{\mathfrak o}(M_{r}^{i})_{\sigma^{i}}= \{A\in {\mathfrak s}{\mathfrak o}(M_{r})\;\mid \; AP = -PA\}.
\]
Moreover, from (\ref{intrixi}),  $\xi$ is determined by
\begin{equation}\label{xicoor}
\begin{array}{lclclcl}
\xi_{U}V & = & p_{\mathcal H}(\nabla_{U}V), & & \xi_{U}X & = & p_{\mathcal V}(\nabla_{U}X),\\[0.4pc]
\xi_{X}U & = & p_{\mathcal H}(\nabla_{X}U), & & \xi_{X}Y & = & p_{\mathcal V}(\nabla_{X}Y).
\end{array}
\end{equation}
\noindent Here and in what follows $U$ and $V$ (resp., $X$ and $Y)$ denote local vector fields of ${\mathcal V}$ (resp., of ${\mathcal H}).$

A smooth distribution $\sigma$ is said to be {\em involutive} if $[\sigma,\sigma]\subset \sigma$ and {\em completely integrable} if for every $x\in M$ there exists an integral submanifold $L$ of $\sigma$ such that $x\in L.$ It follows that if $\sigma$ is completely integrable then it must be involutive. The converse may not hold for the non-regular case (see \cite{Suss}). The involutive condition is equivalent to be completely integrable the maximal regular distribution $\sigma_{r}.$

The {\em second fundamental forms} (symmetric tensors) $h_{\mathcal V}:{\mathcal V}\times{\mathcal V}\to {\mathcal H},$ $h_{\mathcal H}:{\mathcal H}\times{\mathcal H}\to {\mathcal V}$ and the {\em integrability tensors} (skew-symmetric tensors) $A_{\mathcal V}: {\mathcal V}\times {\mathcal V}\to {\mathcal H},$ $A_{\mathcal H}:{\mathcal H}\times {\mathcal H}\to {\mathcal V}$ of $\sigma_{r}$ are defined in terms of $\xi$ by the following formulas:
$$
\begin{array}{lclclcl}
h_{\mathcal V}(U,V) & = & \frac{1}{2}(\xi_{U}V + \xi_{V}U), & A_{\mathcal V}(U,V) & = & \frac{1}{2}(\xi_{U}V -\xi_{V}U),\\[0.4pc]
h_{\mathcal H}(X,Y) & = & \frac{1}{2}(\xi_{X}Y + \xi_{Y}X), & A_{\mathcal H}(X,Y) & = & \frac{1}{2}(\xi_{X}Y -\xi_{Y}X).
\end{array}
$$
\noindent Hence, $A_{\mathcal V}(U,V) = \frac{1}{2}p_{\mathcal H}[U,V]$ and $A_{\mathcal H}(X,Y) = \frac{1}{2}p_{\mathcal V}[X,Y]$ and so $\sigma_{r}$ (resp., the orthogonal distribution $\sigma_{r}^{\bot}$ of $\sigma_{r})$ is completely integrable if and only if $A_{\mathcal V} = 0$ (resp., $A_{\mathcal H} = 0).$ The distribution $\sigma_{r}$ (resp., $\sigma_{r}^{\bot})$ is said to be {\em geodesic} if $h_{\mathcal V} = 0$ (resp., $h_{\mathcal H} = 0).$ It means that all geodesics on $M_{r}$ with initial vector in ${\mathcal V}$ (resp., ${\mathcal H})$ remain in ${\mathcal V}$ (resp., ${\mathcal H})$ for all time. The distribution $\sigma_{r}$ (resp., $\sigma_{r}^{\bot})$ is said to be {\em Riemannian} if $\sigma_{r}$ (resp., $\sigma_{r}^{\bot})$ is integrable and $\sigma_{r}^{\bot}$ (resp., $\sigma_{r})$ is geodesic. If moreover, $\sigma_{r}^{\bot}$ (resp., $\sigma_{r})$ is integrable, we say that $\sigma_{r}$ (resp., $\sigma_{r}^{\bot})$ is a {\em polar distribution}.

Consider the {\em mean curvature vector fields} $H_{\mathcal V}\in {\mathcal H}$ and $H_{\mathcal H}\in {\mathcal V}$ of $\sigma_{r}$ given by $H_{\mathcal V} =\trace h_{\mathcal V}$ and $H_{\mathcal H} = \trace h_{\mathcal H}.$ Then they are locally expressed on each $M_{r}^{i}$ as
\[
H_{\mathcal V} = \sum_{a=1}^{q_{i}}\xi_{E_{a}}E_{a},\;\;\;\; H_{\mathcal H} = \sum_{j=1}^{n-q_{i}}\xi_{E_{q_{i}+j}}E_{q_{i}+j},
\]
where $\{E_{1},\dots ,E_{q_{i}}; E_{q_{i}+1},\dots ,E_{n}\}$ is a local orthonormal frame on $(M_{r}^{i},g_{M_{r}^{i}})$ adapted to $({\mathcal V}^{i},{\mathcal H}^{i}).$ If $H_{\mathcal V}$ (resp., $H_{\mathcal H})$ vanishes the distribution $\sigma_{r}$ (resp., $\sigma_{r}^{\bot})$ is said to be {\em minimal}.

 In the sequel, the following convention for indices is used:  $a,b,\dots \in \{1,\dots ,q_{i}\}$ and $j,k,\dots \in \{1,\dots ,n-q_{i}\},$ for $i\in \{1,\dots ,l\}.$

The {\em second mean curvatures} $\mu_{\mathcal V}$ and $\mu_{\mathcal H}$ are locally defined as
\[
\mu_{{\mathcal V}}  =   \sum_{j}\sum_{a<b}(\xi_{aa}^{j}\xi_{bb}^{j}-\xi_{ab}^{j}\xi_{ba}^{j}),\;\;\;\; \mu_{{\mathcal H}}  =  \sum_{a}\sum_{j<k}(\xi_{jj}^{a}\xi_{kk}^{a}-\xi_{jk}^{a}\xi_{kj}^{a}),
\]
where $\xi_{ab}^{j} = g(\xi_{E_{a}}E_{b},E_{q+j})$ and $\xi_{jk}^{a} = g(\xi_{E_{q+j}}E_{q+k},E_{a}).$ Then,
\begin{equation}\label{mu}
2\mu_{{\mathcal V}} = \|H_{\mathcal V}\|^{2} + \|A_{\mathcal V}\|^{2} - \|h_{\mathcal V}\|^{2},\;\;\;2\mu_{{\mathcal H}} = \|H_{\mathcal H}\|^{2} + \|A_{\mathcal H}\|^{2} - \|h_{\mathcal H}\|^{2}.
\end{equation}

The maximal regular distribution $\sigma_{r}$ (resp., $\sigma_{r}^{\bot})$ is said to be {\em umbilical} if each $\sigma^{i}$ (resp., $(\sigma^{i})^{\bot})$ is umbilical, that is, if  $h_{{\mathcal V}^{i}} = \frac{g_{M_{r}^{i}}(\cdot,\cdot)}{q_{i}}H_{{\mathcal V}^{i}},$ (resp., $h_{{\mathcal H}^{i}} = \frac{g_{M_{r}^{i}}(\cdot,\cdot)}{n-q_{i}}H_{{\mathcal H}^{i}}),$ or equivalently $\xi_{ab}^{j} = -\xi_{ba}^{j}$ and $\xi^{j}_{aa} = \xi^{j}_{bb},$ for all $a\neq b$ (resp., $\xi_{jk}^{a} = -\xi_{kj}^{a}$ and $\xi^{a}_{jj} = \xi^{a}_{kk},$ for all $j\neq k).$
\begin{remark}\label{Rumbilical} {\rm If $\sigma_{r}$ (resp., $\sigma_{r}^{\bot})$ is umbilical then $\mu_{{\mathcal V}}\geq 0$ (resp., $\mu_{{\mathcal H}}\geq 0).$ Moreover, $\mu_{{\mathcal V}^{i}}=0$ (resp., $\mu_{{\mathcal H}^{i}}=0)$ if and only if $\sigma^{i}$ (resp., $(\sigma^{i})^{\bot})$ is integrable and geodesic.}
\end{remark}
\section{Energy of smooth distributions}\indent
 The Kaluza Klein metric $g^{K}$ relative to $(g_{M_{r}},\langle\cdot,\cdot\rangle)$ on the Grassmannian bundle $G(M_{r}),$ where $\langle\cdot,\cdot\rangle$ is the inner product $\langle X,Y\rangle = -\frac{1}{2}\trace XY$ on ${\mathfrak s}{\mathfrak o}(n),$ is defined as follows (see \cite{GDHarm}, \cite{GDM} and \cite{Wood} for more information): Let $TG(M_{r}) = {\rm V}\oplus {\rm H},$ where
\[
{\rm V} = {\rm Ker}\;\pi_{*} = \rho_{*}({\rm Ker}(\pi_{SO(n)})_{*}),\;\;\;\; {\rm H} = \rho_{*}({\rm Ker}\;\omega)
\]
and $\omega$ is the ${\mathfrak s}{\mathfrak o}(n)$-valued connection form of the Levi-Civita connection on ${\mathcal S}{\mathcal O}(M_{r}).$ Because $\rho_{*p}B^{*}_{p} = 0,$ for all $p\in {\mathcal S}{\mathcal O}(M_{r}^{i})$ and $B\in {\mathfrak s}{\mathfrak o}(q_{i})\oplus{\mathfrak s}{\mathfrak o}(n-q_{i}),$ the elements of ${\rm V}_{\rho(p)}$ may be written as $\rho_{*p}A^{*}_{p},$ for some $A$ belonging to the $\langle\cdot,\cdot\rangle$-orthogonal complement  ${\mathfrak m}_{i}$ of ${\mathfrak s}{\mathfrak o}(q_{i})\oplus{\mathfrak s}{\mathfrak o}(n-q_{i})$ in ${\mathfrak s}{\mathfrak o}(n),$ $A^{*}$ being its fundamental vector field on ${\mathcal S}{\mathcal O}(M_{r}).$ Consider the vector bundle ${\mathcal S}{\mathcal O}(M_{r})_{\rho}$ given by the union 
\[
{\mathcal S}{\mathcal O}(M_{r})_{\rho} = \bigcup_{i=1}^{l}({\mathcal S}{\mathcal O}(M_{r}^{i})\times_{S(O(q_{i})\times SO(n-q_{i}))}{\mathfrak m}_{i})
\]
of associated bundles to $\rho_{\mid {\mathcal S}{\mathcal O}(M^{i}_{r})}.$ The map $\iota:{\rm V}\to {\mathcal S}{\mathcal O}(M_{r})_{\rho},$ $\rho_{*p}A^{*}_{p}\mapsto [(p,A)],$ is a vector bundle isomorphism and it may be extended to a type of {\em connection map} ${\mathcal K}: TG(M_{r})\to {\mathcal S}{\mathcal O}(M_{r})_{\rho}$ by saying that ${\mathcal K}(\eta) = 0,$ for all $\eta\in {\rm H},$ and ${\mathcal K}(\eta) = \iota(\eta)$ if $\eta\in {\rm V}.$ The {\em Kaluza-Klein metric} $g^K$ on $G(M_{r})$ is then determined by
\begin{equation}\label{Kaluza}
g^{K}(\eta_{1},\eta_{2}) = g_{M_{r}}(\pi_{*}\eta_{1},\pi_{*}\eta_{2}) + \langle {\mathcal K}(\eta_{1}),{\mathcal K}(\eta_{2})\rangle,
\end{equation}
where $\langle\cdot,\cdot\rangle$ also denotes the fibre metric induced by $\langle\cdot,\cdot\rangle$ restricted to each ${\mathfrak m}_{i}.$ Given a smooth distribution $\sigma$ on $M,$ the pullback bundle $\pi^{*}{\mathfrak s}{\mathfrak o}(M_{r})_{\sigma_{r}}$ by $\pi$ of ${\mathfrak s}{\mathfrak o}(M_{r})_{\sigma_{r}}$ is isomorphic to ${\mathcal S}{\mathcal O}(M_{r})_{\rho}$ (see \cite[Remark 3.2]{GDM}) and a nice property (see proof of \cite[Theorem 3.3]{GDM}) relating ${\mathcal K}$ with the intrinsic torsion establishes that
\[
{\mathcal K}(\sigma_{r*x}X_{x}) = (\sigma_{r}(x),\xi_{X_{x}})\in \pi^{*}{\mathfrak s}{\mathfrak o}(M_{r})_{\sigma_{r}}.
\]
Then, using (\ref{Kaluza}), the pull-back metric $\sigma_{r}^{*}g^{K}$ on $M_{r}$ is given by
\[
(\sigma_{r}^{*}g^{K})(X,Y) = g_{M_{r}}(X,Y) - \frac{1}{2} \trace \xi_{X}\comp \xi_{Y},\;\;\;\; X,Y\in {\mathfrak X}(M_{r}).
\]
Hence, the tensor field $L_{\sigma_{r}}$ on $M_{r}$ can be expressed as $L_{\sigma_{r}} = {\rm Id} + \frac{1}{2}\xi^{t}\comp \xi,$ where $\xi$ is considered as the map $\xi:{\mathfrak X}(M_{r})\to \Gamma^{\infty}({\mathfrak s}{\mathfrak o}(M_{r})_{\sigma_{r}}),$ $\xi(X) = \xi_{X}$ for all $X\in {\mathfrak X}(M_{r}),$ and $\xi^{t}:\Gamma^{\infty}({\mathfrak s}{\mathfrak o}(M_{r})_{\sigma_{r}})\to {\mathfrak X}(M_{r})$ is the adjoint operator of $\xi$ with respect to $g_{M_{r}},$ that is, $g_{M_{r}}(\xi^{t}(\varphi),X) = g_{M_{r}}(\varphi,\xi(X))$ for all $\varphi\in \Gamma^{\infty}({\mathfrak s}{\mathfrak o}(M_{r})_{\sigma_{r}}).$ Then, 
\begin{equation}\label{trace}
\trace L_{\sigma_{r}} = n + \frac{1}{2}\|\xi\|^{2}.
\end{equation} 
Given the set of singular points $M_{s}$ of $\sigma$ as in (\ref{set}), we put $\vec{\varepsilon} = (\varepsilon_{1},\dots ,\varepsilon_{a}; \bar{\varepsilon}_{1},\dots ,\bar{\varepsilon}_{b})\in \RR^{a+b},$ where $\varepsilon_{\alpha},$ $\bar{\varepsilon}_{\beta}>0$ and $\alpha=1,\dots ,a;$ $\beta = 1,\dots ,b.$ Then, for sufficiently small $\varepsilon_{\alpha},$ $\bar{\varepsilon}_{\beta}>0,$ the subset $M(M_{s},\vec{\varepsilon})$ of $M_{r}$ defined as
\[
M(M_{s},\vec{\varepsilon}) = M\setminus \Big(\bigcup_{\alpha=1}^{a}B(x_{\alpha},\varepsilon_{\alpha})\cup \bigcup_{\beta=1}^{b}T(P_{\beta},\bar{\varepsilon}_{\beta})\Big),
\]
where $B(x_{\alpha},\varepsilon_{\alpha})$ denotes the geodesic ball of radius $\varepsilon_{\alpha}$ and center at $x_{\alpha}$ and $T(P_{\beta},\bar{\varepsilon}_{\beta}),$ the tube of radius $\bar{\varepsilon}_{\beta}$ about $P_{\beta},$ is an $n$-dimensional smooth compact manifold. Its boundary $\partial M(M_{s},\vec{\varepsilon})$ is given by the disjoint union
\begin{equation}\label{borde}
\partial M(M_{s},\vec{\varepsilon}) = \bigcup_{\alpha=1}^{a}S(x_{\alpha},\varepsilon_{\alpha})\cup \bigcup_{\beta=1}^{b}P_{\beta}(\bar{\varepsilon}_{\beta}),
\end{equation}
where $S(x_{\alpha},\varepsilon_{\alpha})$ is the geodesic sphere on $M$ of radius $\varepsilon_{\alpha}$ and center at $x_{\alpha}$ and $P_{\beta}(\bar{\varepsilon}_{\beta})$ is the tubular hypersurface about $P_{\beta}$ at a distance $\bar{\varepsilon}_{\beta}$ from $P_{\beta}.$ 

The {\em energy} ${\mathcal E}(\sigma)$ of the distribution $\sigma$ on $M$  is defined as the limit
\[
{\mathcal E}(\sigma) = \lim_{\vec{\varepsilon}\to \vec{0}}{\mathcal E}_{M(M_{s},\vec{\varepsilon})}(\sigma_{r}),
\]
where ${\mathcal E}_{M(M_{s},\vec{\varepsilon})}(\sigma_{r})$ is the energy of the mapping $\sigma_{r}\comp \iota:(M(M_{s},\vec{\varepsilon}), \iota^{*} g_{M_{r}})\to (G(M_{r}),g^{K}),$ $\iota$ being the inclusion of $M(M_{s},\vec{\varepsilon})$ into $M_{r}.$ Note that ${\mathcal E}(\sigma)$ may be infinite. From (\ref{energyf}) and (\ref{trace}), we have
\[
{\mathcal E}_{M(M_{s},\vec{\varepsilon})}(\sigma_{r}) = \frac{n}{2}{\rm Vol}(M(M_{s},\vec{\varepsilon})) + \frac{1}{4}\int_{M(M_{s},\vec{\varepsilon})}\|\xi\|^{2}dv_{M(M_{s},\vec{\varepsilon})}.
\]
For an arbitrary almost continuous function $f:M_{r}\to \RR$ on $M_{r},$ we establishes that
\[
\int_{M}fdv_{M} = \lim_{\vec{\varepsilon}\to \vec{0}}\int_{M(M_{s},\vec{\varepsilon})}f dv_{M(M_{s},\vec{\varepsilon})},
\]
if the limit exists. Therefore, the energy ${\mathcal E}(\sigma)$ of $\sigma$ takes the form
\[
{\mathcal E}(\sigma) = \frac{n}{2}{\rm Vol}(M,g) + \frac{1}{4}\int_{M}\|\xi\|^{2}dv_{M}.
\]
The relevant part of this formula,
$$
B(\sigma) = \frac{1}{4}\int_{M}\|\xi\|^{2}dv_{M},
$$
will be called the {\em total bending} of the distribution. Because $\|\xi\|^{2}\geq 0,$ $B(\sigma)$ is well-defined (may be infinite) and it is zero if and only if $\xi$ vanishes on whole $M_{r}$ or, equivalently, the distributions $\sigma_{r}$ and $\sigma_{r}^{\bot}$ are both geodesic and integrable. Note that ${\mathcal E}_{M(M_{s},\vec{\varepsilon})}(\sigma_{r}) = {\mathcal E}_{M(M_{s},\vec{\varepsilon})}(\sigma_{r}^{\bot})$ and so ${\mathcal E}(\sigma)$ can be also given by ${\mathcal E}(\sigma)= \lim_{\vec{\varepsilon}\to \vec{0}}{\mathcal E}_{M(M_{s},\vec{\varepsilon})}(\sigma_{r}^{\bot}).$

Let $\Sigma_{\mathcal V}$ and $\Sigma_{\mathcal H}$ be the smooth functions on $M_{r},$ locally defined on the domain in $M_{r}^{i},$ for $i = 1,\dots ,l,$ of a $({\mathcal V}_{i},{\mathcal H}_{i})$-adapted frame $\{E_{a};E_{q_{i}+j}\},$ as
\[
\Sigma_{\mathcal V} = \sum_{a,b; j}(\xi_{ab}^{j})^{2},\;\;\;\Sigma_{\mathcal H} = \sum_{a;jk}(\xi_{jk}^{a})^{2}.
\]
Then, from (\ref{xicoor}), we have $\|\xi\|^{2} = 2(\Sigma_{\mathcal V} + \Sigma_{\mathcal H}).$ Hence, 
\begin{equation}\label{Bendingnew}
B(\sigma) = \sum_{i}(B^{\mathcal V}(\sigma^{i}) + B^{\mathcal H}(\sigma^{i})),
\end{equation}
where,
\[
B^{\mathcal V}(\sigma^{i}) = \frac{1}{2}\int_{\bar{M_{r}^{i}}}\Sigma_{\mathcal V} dv_{\bar{M_{r}^{i}}},\;\;\;\;\;\;B^{\mathcal H}(\sigma^{i}) = \frac{1}{2}\int_{\bar{M_{r}^{i}}}\Sigma_{\mathcal H} dv_{\bar{M_{r}^{i}}}
\]
and $\bar{M_{r}^{i}}$ is the closure of $M_{r}^{i}.$

\begin{lemma}\label{Sigma} If $q_{i}\geq 2,$ then
\begin{equation}\label{Sigma1}
B^{\mathcal V}(\sigma^{i})\geq \frac{1}{q_{i}-1}\int_{\bar{M}^{i}_{r}}\mu_{{\mathcal V}^{i}}\;dv_{\bar{M}_{r}^{i}}
\end{equation}
and, if $n-q_{i}\geq 2,$ then
\begin{equation}\label{Sigma2}
B^{\mathcal H}(\sigma^{i}) \geq\frac{1}{n-q_{i}-1}\int_{\bar{M}_{r}^{i}}\mu_{{\mathcal H}^{i}}\;dv_{\bar{M}_{r}^{i}}.
\end{equation}
Moreover, the equality in {\rm (\ref{Sigma1})} {\rm (}resp., in {\rm (\ref{Sigma2})}{\rm )} holds if and only if
\begin{enumerate}
\item[{\rm (a)}] for $q_{i} = 2$ {\rm (}resp., $n-q_{i} = 2),$ $\sigma^{i}$ {\rm (}resp., $(\sigma^{i})^{\bot})$ is umbilical;
\item[{\rm (b)}] for $q_{i}\geq 3$ {\rm (}resp., $n-q_{i}\geq 3),$ $\sigma^{i}$ {\rm (}resp., $(\sigma^{i})^{\bot})$ is umbilical and integrable.
\end{enumerate}
\end{lemma}
\begin{proof} For each $j=1,\dots ,n-q_{i},$ observe that
$$
\begin{array}{lcl}
\sum_{a<b}(\xi^{j}_{aa} - \xi_{bb}^{j})^{2} & = & (q_{i}-1)\sum_{a}(\xi_{aa}^{j})^{2} - 2\sum_{a<b}\xi_{aa}^{j}\xi_{bb}^{j},\\[0.5pc]
\sum_{a<b}(\xi_{ab}^{j} + \xi_{ba}^{j})^{2} & = & \sum_{a\neq b}(\xi_{ab}^{j})^{2} + 2\sum_{a<b}\xi_{ab}^{j}\xi_{ba}^{j}.
\end{array}
$$
Then, summing these two equations and using that
$$
\sum_{a,b}(\xi_{ab}^{j})^{2} = \sum_{a}(\xi_{aa}^{j})^{2} + \sum_{a\neq b}(\xi_{ab}^{j})^{2},
$$
we obtain, for $q_{i}\geq 2,$ that
$$
\begin{array}{lcl}
\displaystyle\sum_{a,b}(\xi_{ab}^{j})^{2} & = & \frac{1}{q_{i}-1}\Big \{ \displaystyle\sum_{a<b}(\xi_{aa}^{j} - \xi_{bb}^{j})^{2} + \sum_{a<b}(\xi_{ab}^{j} + \xi_{ba}^{j})^{2}\\[0.4pc]
 & & \hspace{0.2cm} + (q_{i}-2)\displaystyle\sum_{a\neq b}(\xi_{ab}^{j})^{2} + 2\sum_{a<b}(\xi_{aa}^{j}\xi_{bb}^{j} - \xi_{ab}^{j}\xi_{ba}^{j})\Big \}.
  \end{array}
 $$
In similar way, for $n-q_{i} \geq 2$ and for each $a=1,\dots ,q_{i},$
$$
\begin{array}{lcl}
\displaystyle\sum_{j,k}(\xi_{jk}^{a})^{2} & = & \frac{1}{n-q_{i}-1}\Big \{ \displaystyle\sum_{j<k}(\xi_{jj}^{a} - \xi_{kk}^{a})^{2} + \sum_{j<k}(\xi_{jk}^{a} + \xi_{kj}^{a})^{2}\\[0.4pc]
 & & \hspace{0.2cm} + (n-q_{i}-2)\displaystyle\sum_{j\neq k}(\xi_{jk}^{a})^{2} + 2\sum_{j<k}(\xi_{jj}^{a}\xi_{kk}^{a} - \xi_{jk}^{a}\xi_{kj}^{a})\Big \}.
 \end{array}
$$
Then, we get the lemma.
\end{proof}
Hence, using (\ref{Bendingnew}), we prove the following result.
\begin{proposition}\label{pdesigual} If $n\geq 3,$ then
\begin{equation}\label{Bendingeq}
B(\sigma) \geq \sum_{i = 1}^{l}\int_{\bar{M}_{r}^{i}}\Big ( \frac{\mu_{{\mathcal V}^{i}}}{q_{i} -1} + \frac{\mu_{{\mathcal H}^{i}}}{n-q_{i} -1}\Big ) dv_{{\bar M}_{r}^{i}}.
\end{equation}
The equality holds if and only if the distributions $\sigma^{i}$ {\rm (}resp., $(\sigma^{i})^{\bot}),$ for $i = 1,\dots ,l,$ are:
\begin{enumerate}
\item[{\rm (i)}] geodesic, if $q_{i} = 1$ {\rm (}resp., $q_{i} = n-1);$
\item[{\rm (ii)}] umbilical, if $q_{i} = 2$ {\rm (}resp., $q_{i} = n-2);$
\item[{\rm (iii)}] umbilical and integrable, if $3\leq q_{i}\leq n-1$ {\rm (}resp., $n-q_{i}\geq 3).$
\end{enumerate}
\end{proposition}

\begin{remark}{\rm If $q_{i} = 1$ (resp., $q_{i} = n-1)$ then $\mu_{{\mathcal V}^{i}} = 0$ (resp., $\mu_{{\mathcal H}^{i}} = 0).$ In formula (\ref{Bendingeq}) and for this case, the quotient $\mu_{{\mathcal V}^{i}}/(q_{i} -1)$ (resp., $\mu_{{\mathcal H}^{i}}/(n-q_{i}-1))$ is supposed to be zero.}
\end{remark}

\section{An integral formula for almost regular distributions}\indent
The {\em mixed scalar curvature} $s_{\rm mix}(\sigma)$ of a smooth distribution $\sigma$ is the function on the set $M_{r}= \cup_{i=1}^{l}M_{r}^{i}\subset M$ of its regular points, locally defined on each $M_{r}^{i}$ by
\[
s_{\rm mix}(\sigma) = \sum_{a;j} g_{M_{r}^{i}}(R_{E_{a}E_{q_{i}+j}}E_{a},E_{q+j}),
\]
where $\{E_{a};E_{q_{i}+j}\}$ is an adapted local orthonormal frame of $({\mathcal V}^{i},{\mathcal H}^{i})$ on $(M_{r}^{i},g_{M_{r}^{i}})$ and $R$ is the Riemannian curvature tensor taken with the sign convention $R_{XY} = \nabla_{[X,Y]} - [\nabla_{X},\nabla_{Y}],$ for all $X,Y\in {\mathfrak X}(M).$ In particular, if $\sigma$ is an one-dimensional (resp., a codimension one) almost regular distribution, the mixed scalar curvature is locally expressed as ${\rm Ric}(V,V),$ where $V$ is a local unit vector field belonging to ${\mathcal V}$ (resp., to ${\mathcal H})$ and ${\rm Ric}$ is the Ricci tensor of $(M,g).$ In \cite{Walczak} it is proved the formula
\[
s_{\rm mix}(\sigma) = {\rm div}(H_{\mathcal V} + H_{\mathcal H}) + \|H_{\mathcal V}\|^{2} + \|H_{\mathcal H}\|^{2}  + \|A_{\mathcal V}\|^{2} + \|A_{\mathcal H}\|^{2}-  \|h_{\mathcal V}\|^{2} - \|h_{\mathcal H}\|^{2},
\]
which, from (\ref{mu}), can be written as
\begin{equation}\label{divergencia}
s_{\rm mix}(\sigma) = {\rm div}(H_{\mathcal V}+H_{\mathcal H}) + 2(\mu_{\mathcal V} + \mu_{{\mathcal H}}).
\end{equation}

In what follows, we shall suppose that $M_{s},$ given as in (\ref{set}), satisfies the additional condition that $\dim P_{\beta}\leq n-2$ for each $\beta=1,\dots ,b.$ (We again note that $M_{s}$ may be empty). From the next lemma, this implies that $\sigma$ must be almost regular. 

We say that $\sigma$ is {\em trivial} if it is an $n$-dimensional almost regular distribution. Note that the total bending of a trivial distribution is zero.
\begin{lemma}\label{lconnected} If $\dim P_{\beta}\leq n-2$ for all $\beta\in \{1,\dots ,b\},$ then $M_{r}$ is connected. Moreover, if $\sigma$ is not a trivial distribution, the converse holds.
\end{lemma}
\begin{proof} Suppose that $M_{r}$ is not connected. Then there exist two disjoint nonempty open subsets $M_{r}'$ and $M_{r}''$ such that $M_{r} = M_{r}'\cup M_{r}''.$ Because $M_{r}$ is dense in $M,$ we have that $M = \bar{M}_{r}'\cup \bar{M}_{r}''$ and, from the connectedness of $M,$ it follows that the intersection of their frontiers, ${\rm Fr}(M_{r}')\cap {\rm Fr}(M_{r}''),$ is nonempty. Since ${\rm Fr}(M_{r}')\cap {\rm Fr}(M_{r}'')$ is a subset of $M_{s},$ each one of its connected components is a regular submanifold $P$ (may be a point) of $M.$ Let $({\mathcal U},\varphi=(x^{1},\dots ,x^{n}))$ be a sufficiently small coordinate neighborhood adapted to $P.$ Then, ${\mathcal U}\cap P = {\mathcal U}\cap ({\rm Fr}(M_{r}')\cap {\rm Fr}(M_{r}''))$ and it is a $k$-dimensional slice for some $k\leq n-1.$ Now, taking into account that ${\mathcal U}\cap M_{r}'$ and ${\mathcal U}\cap M_{r}''$ are non-empty open subsets in ${\mathcal U},$ $k$ must be equals to $n-1$ and so, $\dim P=n-1.$ This contradicts the hypothesis of the lemma.

For the converse, suppose that $M_{r}$ is connected and $\dim P_{\beta} = n-1,$ for some $\beta\in \{1,\dots ,b\}.$ Then $d$ takes the value $n$ on whole $M_{r}$ and so $\sigma$ must be trivial.
\end{proof}
\begin{remark}{\rm If $\dim P_{\beta} = n-1,$ for some $\beta\in \{1,\dots ,b\},$ then there exist connected components of $M_{r}$ where $d$ is constant equals to $n.$ From Lemma \ref{lconnected}, any non-almost regular distribution satisfies this condition.}
\end{remark}
Next, we prove the following integral formula, which extends the one given in \cite[Lemma 2.6]{BGV} for unit vectors with singularities.

\begin{theorem}\label{div} If $B(\sigma)<\infty,$ then
\begin{equation}\label{int1}\int_{M}s_{\rm mix}(\sigma)\;dv_{M} = 2\int_{M}(\mu_{\mathcal V}+ \mu_{\mathcal H})\;dv_{M}.
\end{equation}
\end{theorem}
For the proof, we first need the following lemmas.

\begin{lemma}\label{lgeneral} {\rm \cite[Lemma 2.4]{BGV}} Let $f:M_{r} \to [0,\infty[$ be an almost continous function on $M_{r}$ and suppose that
 \begin{enumerate}
 \item[{\rm (i)}] there exists a point $x\in M_{s}$ such that
 \[
 \liminf_{r\to 0^+}\int_{S(x,r)} f\; dv_{S(x,r)}>0\;\;\;{\rm or,}
 \]
 \item[{\rm (ii)}] there exists an embedded submanifold $P\subset M_{s},$ $\dim P\leq n-2,$ such that
 \[
 \liminf_{r\to 0^+}\int_{P(r)} f\; dv_{P(r)}>0,
 \]
 then
 \[
 \int_{M}f^{2}\; dv_{M} = \infty.
 \]
 \end{enumerate}
 \end{lemma}

 \begin{lemma}\label{lnew} There exists a constant $C_{n},$ which only depends on $n,$ such that
 \[
 \|H_{\mathcal V} + H_{\mathcal H}\|\leq C_{n}\|\xi\|.
 \]
 \end{lemma}
 \begin{proof} Taking a $({\mathcal V}_{i},{\mathcal H}_{i})$-adapted local orthonormal frame $\{E_{a};E_{q_{i}+j}\}$ in each open subset $M_{r}^{i},$ one gets
 $$
 \begin{array}{lcl}
  \|H_{\mathcal V} + H_{\mathcal H}\|^{2} & = & \|H_{\mathcal V}\|^{2} + \|H_{\mathcal H}\|^{2} = \sum_{j}\Big (\sum_{a}(\xi_{aa}^{j})^{2}+ 2\sum_{a<b}\xi_{aa}^{j}\xi_{bb}^{j}\Big )\\[0.4pc]
  & & \hspace{0.2cm} +\sum_{a}\Big (\sum_{j} (\xi_{jj}^{a})^{2} + 2\sum_{j<k}\xi_{jj}^{a}\xi_{kk}^{a}\Big ).
\end{array}
$$
 Hence, because $\|\xi\|^{2} = 2(\Sigma_{\mathcal V} + \Sigma_{\mathcal H}),$
$$
\begin{array}{lcl}
\|\xi\|^{2} - \|H_{\mathcal V} + H_{\mathcal H}\|^{2} & = & (2\Sigma_{\mathcal V} - \|H_{\mathcal V}\|^{2}) + (2\Sigma_{\mathcal H} - \|H_{\mathcal H}\|^{2})\\[0.4pc]
 & =& \sum_{j}\Big (\sum_{a}(\xi_{aa}^{j})^{2} + 2\sum_{a\neq b}(\xi_{ab}^{j})^{2} - 2\sum_{a<b}\xi_{aa}^{j}\xi_{bb}^{j}\Big )\\[0.4pc]
&  & + \sum_{a}\Big (\sum_{j}(\xi_{jj}^{a})^{2} + 2\sum_{j\neq k}(\xi_{jk}^{a})^{2} - 2\sum_{j<k}\xi_{jj}^{a}\xi_{kk}^{a}\Big ).
\end{array}
$$
Then,
$$
\begin{array}{lcl}
(q_{i} + 1)\Sigma_{\mathcal V} - \|H_{\mathcal V}\|^{2} & = & (q_{i}-1)\Sigma_{\mathcal V} + (2\Sigma_{\mathcal V}-\|H_{\mathcal V}\|^{2})\\[0.4pc]
 & = & \sum_{j}\Big (\sum_{a<b}(\xi_{aa}^{j}-\xi_{bb}^{j})^{2} + 2\sum_{a<b}\xi_{aa}^{j}\xi_{bb}^{j} + (q_{i}-1)\sum_{a\neq b}(\xi_{ab}^{j})^{2}\Big )\\[0.4pc]
 & & + \sum_{j}\Big (\sum_{a}(\xi_{aa}^{j})^{2} + 2\sum_{a\neq b}(\xi_{ab}^{j})^{2} - 2\sum_{a<b}\xi_{aa}^{j}\xi_{bb}^{j}\Big )\\[0.4pc]
 & = & \sum_{j}\Big (\sum_{a<b}(\xi_{aa}^{j}-\xi_{bb}^{j})^{2} + (q_{i}+1)\sum_{a\neq b}(\xi_{ab}^{j})^{2} + \sum_{a}(\xi_{aa}^{j})^{2}\Big ) \geq 0.
 \end{array}
 $$
In the same way, we get
$$
\begin{array}{lcl}
(n-q_{i} + 1)\Sigma_{\mathcal H} - \|H_{\mathcal H}\|^{2} & = &  \sum_{a}\Big (\sum_{j<k}(\xi_{jj}^{a}-\xi_{kk}^{a})^{2} \\[0.4pc]
& & \hspace{0.2cm} + (n-q_{i}+1)\sum_{j\neq k}(\xi_{jk}^{a})^{2} + \sum_{a}(\xi_{jj}^{a})^{2}\Big )\geq 0.
 \end{array}
$$
Therefore, we have proved on $M_{r}^{i}$ that $\frac{(q_{i}+1)(n-q_{i}+1)}{2}\|\xi\|^{2} - \|H_{\mathcal V} + H_{\mathcal H}\|^{2}\geq 0.$ Because the function $f(x) = \frac{(x+1)(n-x+1)}{2}$ has a maximum at $x = \frac{n}{2},$ the inequality $\frac{(n+2)^{2}}{8}\|\xi\|^{2} - \|H_{\mathcal V} + H_{\mathcal H}\|^{2}\geq 0$ holds on whole $M_{r}.$ Thus, taking $C_{n} = \frac{(n+2)\sqrt{2}}{4},$ the lemma is proved.
\end{proof}

\noindent {\em Proof of Theorem {\rm \ref{div}}.} Denote by $N$ the unit outward normal vector field to $\partial M(M_{s},\vec{\varepsilon}).$ Then, from (\ref{divergencia}) and the divergence theorem, we obtain
$$
\begin{array}{lcl}
\left| \int_{M(M_{s},\vec{\varepsilon})}(s_{\rm mix} - 2(\mu_{\mathcal V} + \mu_{\mathcal H}))\;dv_{M(M_{s},\vec{\varepsilon})}\right| &  \leq & \int_{\partial M(M_{s},\vec{\varepsilon})}|g(H_{\mathcal V} + H_{\mathcal H},N)|\;dv_{\partial M(M_{s},\vec{\varepsilon})}\\[0.7pc]
 & \leq  & \int_{\partial M(M_{s},\vec{\varepsilon})}\|H_{\mathcal V} + H_{\mathcal H}\|\;dv_{\partial M(M_{s},\vec{\varepsilon})}.
\end{array}
$$
Hence, using Lemma \ref{lnew}, it follows that
$$
\begin{array}{l}
\left| \int_{M(M_{s},\vec{\varepsilon})}(s_{\rm mix} - 2(\mu_{\mathcal V} + \mu_{\mathcal H}))\;dv_{M(M_{s},\vec{\varepsilon})}\right| \leq C_{n}\int_{\partial M(M_{s},\vec{\varepsilon})}\|\xi\|\;dv_{\partial M(M_{s},\vec{\varepsilon})}\\[0.6pc]
\hspace{0.3cm}= C_{n}\left(\sum_{\alpha=1}^{a}\int_{S(x_{\alpha},\varepsilon_{\alpha})}\|\xi\;\|dv_{S(x_{\alpha},\varepsilon_{\alpha})} + \sum_{\beta=1}^{b}\int_{P_{\beta}(\bar{\varepsilon}_{\beta})}\|\xi\|\;dv_{P_{\beta}(\bar{\varepsilon}_{\beta})}\right).
\end{array}
$$
Now, putting $f = \|\xi\|$ in Lemma \ref{lgeneral}, we have 
\[
\liminf_{r_{\alpha}\to 0^{+}}\int_{S(x_{\alpha},\varepsilon_{\alpha})}\|\xi\|\;dv_{S(x_{\alpha},\varepsilon_{\alpha})} = \liminf_{\bar{r}_{\beta}\to 0^{+}}\int_{P_{\beta}(\bar{\varepsilon}_{\beta})}\|\xi\|\;dv_{P_{\beta}(\bar{\varepsilon}_{\beta})} = 0,
\]
for $\alpha = 1,\dots ,a,$ $\beta = 1,\dots ,b.$ This implies that the integral $\int_{M}(\mu_{\mathcal V} + \mu_{\mathcal H})\;dv_{M}$ converges and the result follows.

\begin{remark}{\rm For codimension one almost regular distributions on a compact Einstein manifold $(M,g),$ the integral equation (\ref{int1}) takes the form
\begin{equation}\label{int2}
\int_{M}\mu_{\mathcal V}dv_{M} = \frac{\tau}{2n}{\rm Vol}(M,g),
\end{equation}
where $\tau$ is the scalar curvature of $(M,g).$ In \cite[Remark 3.5]{BGV} it is proved that the total bending of the orthogonal distribution to the radial vector field around a point in the complex projective space ${\mathbb C}P^{m}(\lambda)$ is infinite and that (\ref{int2}) is not satisfied. Hence, the assumption $B(\sigma)<\infty$ in Theorem \ref{div} can not be omitted.}
\end{remark}

\section{The energy functional of almost regular distributions}\indent
Let $\sigma$ be a $q$-dimensional almost regular distribution on an $n$-dimensional compact Riemannian manifold $(M,g).$ Suppose, as in previous section, that the set of singular points $M_{s}$ of $\sigma$ is of the form (\ref{set}) and $\dim P_{\beta}\leq n-2$ for each $\beta=1,\dots ,b.$ Next, we extend, using Theorem \ref{div}, some results about the total bending of unit vector fields with singularities given in \cite{BGV} and \cite{BW}. For the two-dimensional case, we have: 

\begin{theorem}\label{Tori} Tori are the unique compact oriented surfaces in $\RR^{3}$ admitting an one-dimensional almost regular distribution $\sigma,$ with a finite set of singular points, which may be empty, such that $B(\sigma)<\infty.$
\end{theorem}
\begin{proof} Since the mixed curvature of $\sigma$ on a surface $M$ coincides with the restriction of its Gauss curvature $K$ to $M_{r},$ it follows from Theorem \ref{div} that
\[
\int_{M}K\;dv_{M} = 0.
\]
Then, from the Gauss-Bonnet Theorem and the classification of oriented compact surfaces, $M$ has to be a torus. For the converse, consider the rotational torus $T^2$ in $\RR^{3}$ given by
\[
T^{2} = \{(x,y,z) = ((R + r\cos \theta)\cos \varphi, (R + r\cos \theta)\sin \varphi, r\sin \theta )\},\;\; 0<r<R.
\]
The level sets of the (isoparametric) function $f:T^{2}\to \RR,$ $f(x,y,z) = z,$ determine a regular (Riemannian) foliation ${\mathcal F}.$ A local orthonormal frame adapted to the corresponding tangent distribution $\sigma$ of ${\mathcal F}$ is given by the vector fields $E_{1} = \frac{1}{(R + r\cos\theta)}\frac{\partial}{\partial \varphi},$ $E_{2} = \frac{1}{r}\frac{\partial}{\partial \theta}.$
Since $[E_{1},E_{2}] = -\frac{\sin \theta}{R + r\cos\theta}E_{1},$ the Koszul formula implies that  $\xi_{11}^{2} = \frac{\sin\theta}{R + r\cos\theta}$ and $\xi_{22}^{1} = 0.$ Hence,
\[
B(\sigma) = \frac{1}{2}\int_{0}^{2\pi}\Big (\int_{0}^{2\pi}\frac{\sin^{2}\theta}{(R+r\cos\theta)^{2}}d\theta\Big )d\varphi\leq 2\left ( \frac{\pi}{R-r}\right )^{2}< \infty.
\]
\end{proof}
For $n\geq 3,$ we extend the results in \cite[Theorem 2.2]{BGV} and in \cite[Theorem 1]{BW}. 
\begin{theorem}\label{tbending} Let $\sigma$ be a $q$-dimensional almost regular distribution on an $n$-dimensional, $n\geq 3,$ compact Riemannian manifold $(M,g).$ We have the following cases:
\begin{enumerate}
\item[{\rm (I)}] If $q = 1$ {\rm (}resp., $n-q = 1),$ then
\begin{equation}\label{q1}
B(\sigma)  \geq \frac{1}{2(n-2)}\int_{M}s_{\rm mix}(\sigma)\;dv_{M}.
\end{equation}
The equality holds if and only if one of the following conditions is satisfied:
\begin{enumerate}
\item[{\rm (i)}] if $n = 3,$ $\sigma_{r}$ {\rm (}resp., $\sigma_{r}^{\bot})$ is geodesic and $\sigma_{r}^{\bot}$ {\rm (}resp., $\sigma_{r})$ is umbilical.
\item[{\rm (ii)}] if $n\geq 4,$ $\sigma_{r}^{\bot}$ {\rm (}resp., $\sigma_{r})$ defines an umbilical Riemannian foliation on $M_{r}.$
\end{enumerate}
\item[{\rm (II)}] If $q=\frac{n}{2},$ then
\[
B(\sigma)  \geq \frac{1}{n-2}\int_{M}s_{\rm mix}(\sigma)\;dv_{M}.
\]
The equality holds if and only if one of the following conditions is satisfied:
\begin{enumerate}
\item[{\rm (i)}] if $q=2$ $(n = 4),$ $\sigma_{r}$ and $\sigma_{r}^{\bot}$ are umbilical distributions.
\item[{\rm (ii)}] if $q\geq 3,$ $\sigma_{r}$ and $\sigma_{r}^{\bot}$ define umbilical foliations on $M_{r}.$
\end{enumerate}
\item[{\rm (III)}] If $1<q< \frac{n}{2}$ $(n\geq 5)$ and $\sigma_{r}$ is an umbilical distribution {\rm (}resp., $\frac{n}{2}<q< n-1$ and $\sigma_{r}^{\bot}$ umbilical), then
\[
B(\sigma)  \geq \frac{1}{2(n-q-1)}\int_{M}s_{\rm mix}(\sigma)\;dv_{M}\;\;\;\;(\mbox{resp.,}\;B(\sigma)  \geq \frac{1}{2(q-1)}\int_{M}s_{\rm mix}(\sigma)\;dv_{M}).
\]
The equality holds if and only if $\sigma_{r}^{\bot}$ {\rm (}resp., $\sigma_{r})$ defines an umbilical polar foliation on $M_{r}.$
\end{enumerate}
\end{theorem}
\begin{proof} We can suppose that $B(\sigma)< \infty$ because if $B(\sigma) = \infty,$ there is nothing to prove. Moreover, from (\ref{Bendingeq}), one gets
\begin{equation}\label{almost}
B(\sigma) \geq \int_{M}\Big ( \frac{\mu_{\mathcal V}}{q -1} + \frac{\mu_{\mathcal H}}{n-q-1}\Big ) dv_{M}.
\end{equation}

If $q = 1$ then $\mu_{\mathcal V} = 0.$ Hence, (\ref{almost}) and Theorem \ref{div} imply that
\[
B(\sigma)\geq \frac{1}{n-2}\int_{M}\mu_{\mathcal H}\;dv_{M} = \frac{1}{2(n-2)}\int_{M}s_{\rm mix}(\sigma)\;dv_{M}.
\]
From Proposition \ref{pdesigual}, we get (i) and (ii) in (I). In similar way, for $q = n-1$ is also proved.

If $q= \frac{n}{2},$ the inequality (\ref{almost}) together Theorem \ref{div} imply that
\[
B(\sigma)\geq \frac{1}{q-1}\int_{M}(\mu_{\mathcal V} + \mu_{\mathcal H})dv_{M} = \frac{1}{n-2}\int_{M}s_{\rm mix}(\sigma)\;dv_{M}.
\]
Moreover, applying Proposition \ref{pdesigual}, the case (II) follows.

Finally, suppose that $1<q< \frac{n}{2}$ and $\sigma_{r}$ is an umbilical distribution. Then, from Remark \ref{Rumbilical}, (\ref{almost}) and Theorem \ref{div}, we obtain
$$
\begin{array}{lcl}
B(\sigma) & \geq & \displaystyle\frac{1}{n-q-1}\displaystyle\int_{M}(\mu_{\mathcal V} + \mu_{\mathcal H})\;dv_{M} + \frac{n-2q}{(q-1)(n-q-1)}\int_{M}\mu_{\mathcal V}\;dv_{M}\\[0.4pc]
 & \geq & \displaystyle\frac{1}{n-q-1}\int_{M}(\mu_{\mathcal V} + \mu_{\mathcal H})\;dv_{M} = \frac{1}{2(n-q-1)}\int_{M}s_{\rm mix}(\sigma)\;dv_{M}.
 \end{array}
 $$
For the equality in this expression, we use again Remark \ref{Rumbilical} and Proposition \ref{pdesigual}. This proves (III). For the case $\frac{n}{2}<q< n-1$ and $\sigma_{r}^{\bot}$ umbilical we use similar arguments.
\end{proof}
 For one-dimensional or codimension one almost regular distributions on compact Einstein manifolds, the inequality (\ref{q1}) may also written as
\[
B(\sigma) \geq \frac{\tau}{2n(n-2)}{\rm Vol}(M,g).
\]
In \cite{Chen} it is proved that if an irreducible symmetric space $M$ admits a totally umbilical hypersurface $N$ then both $M$ and $N$ are of constant curvature. Moreover, there is no totally umbilical submanifold of codimension less than ${\rm rank}\; M -1.$ Hence, Theorem \ref{tbending} leads to the following corollary:
\begin{corollary} Let $(M,g)$ be an $n$-dimensional compact, irreducible symmetric space equipped with a $q$-dimensional almost regular distribution $\sigma.$ Then, we have:
\begin{enumerate}
\item[{\rm (i)}] If $q=1$ or $q= n-1,$ $n\geq 4$ and $(M,g)$ has nonconstant sectional curvature,
\[
B(\sigma) > \frac{\tau}{2n(n-2)}{\rm Vol}(M,g).
\]
\item[{\rm (ii)}] If $n=2q\geq 6$ and ${\rm rank}\;M> q +1,$
\[
B(\sigma) > \frac{1}{n-2}\int_{M}s_{\rm mix}(\sigma)\;dv_{M}.
\]
\end{enumerate}
\end{corollary}

Euclidean spheres, together projective spaces ${\mathbb K}P^{m},$ where ${\mathbb K} =\RR,$ ${\mathbb C}$ or ${\mathbb H},$ and the Cayley plane ${\mathbb C}aP^{2},$ are all the compact rank one symmetric spaces. Denote by $S^{n}(\lambda)$ and $\RR P^{m}(\lambda)$ the sphere and the real projective space of constant curvature $\lambda,$ by ${\mathbb C}P^{m}(\lambda)$ the complex projective space with constant holomorphic sectional curvature $c = 4\lambda,$ by ${\mathbb H}P^{m}(\lambda)$ the quaternionic projective space with constant quaternionic sectional curvature $c = 4\lambda$ and by ${\mathbb C}aP^{2}(\lambda)$ the Cayley plane with constant Cayley sectional curvature $c = 4\lambda.$ We denote by $J$ or $J_{1}$ the complex structure if ${\mathbb K} = {\mathbb C},$ and by $\{J_{1},\dots J_{\nu}\}$ (a local basis of) the quaternionic K\"ahler structure or the Cayley structure, depending on wether ${\mathbb K} = {\mathbb H}$ and $\nu = 3$ or ${\mathbb K} = {\mathbb C}a$ and $\nu = 7.$ 

A distribution $\sigma$ on ${\mathbb K}P^{m}(\lambda),$ ${\mathbb K}\neq \RR,$ is said to be {\em invariant} if $J_{s}{\mathcal V}\subset {\mathcal V},$ for all $s= 1,\dots, \nu.$ Then also $J_{s}{\mathcal H}\subset {\mathcal H}$ and, for each $i = 1,\dots ,l,$ $\dim {\mathcal V}^{i} = (\nu+ 1)\kappa_{i},$ for some $\kappa_{i}\leq m.$
\begin{corollary}\label{pinv} Let $\sigma$ be an invariant $q$-dimensional almost regular distribution on a compact projective space $(M,g) ={\mathbb K}P^{m}(\lambda),$ ${\mathbb K}\neq \RR.$ We have:
\begin{enumerate}
\item[{\rm (I)}] If $q= \frac{n}{2},$ then
\[
B(\sigma) \geq \frac{q^{2}}{n-2}\lambda{\rm Vol}({\mathbb K}P^{m}(\lambda)).
\]
The equality holds if and only if one of the conditions is satisfied:
\begin{enumerate}
\item[{\rm (i)}] if $(M,g) = {\mathbb C}P^{2}(\lambda),$ $\sigma_{r}$ and $\sigma_{r}^{\bot}$ are umbilical distributions.
\item[{\rm (ii)}] if $q\geq 3,$ $\sigma_{r}$ and $\sigma_{r}^{\bot}$ define umbilical foliations on $M_{r}.$
\end{enumerate}
\item[{\rm (II)}] If $1<q< \frac{n}{2}$ $(n\geq 5)$ and $\sigma_{r}$ is an umbilical distribution {\rm (}resp., $\frac{n}{2}<q< n-1$ and $\sigma_{r}^{\bot}$ umbilical{\rm )}, then
\[
B(\sigma) \geq \frac{q(n-q)}{2(n-q-1)}\lambda{\rm Vol}({\mathbb K}P^{m}(\lambda))\;\;\;\;(\mbox{resp.,}\;
B(\sigma) \geq \frac{q(n-q)}{2(q-1)}\lambda{\rm Vol}({\mathbb K}P^{m}(\lambda))).
\]
The equality holds if and only if $\sigma_{r}^{\bot}$ {\rm (}resp., $\sigma_{r})$ defines an umbilical polar foliation on $M_{r}.$
\end{enumerate}
\end{corollary}
\begin{proof} The Jacobi operator $R_{u}$ on ${\mathbb K}P^{m}(\lambda),$ ${\mathbb K}\neq \RR,$ defined by $R_{u} = R_{u\cdot}u$ for a unit vector $u,$ satisfies \cite{Vanhecke}
\[
R_{u}J_{s}u = 4\lambda J_{s}u,\;\;\; R_{u}X = \lambda X,\;\; X\in \{u,J_{s}u\}^{\bot},
\]
for $s = 1,\dots ,\nu.$ Then, taking an adapted local orthonormal frame of the pair $({\mathcal V},{\mathcal H})$ associated to $\sigma$ on ${\mathbb K}P^{m}(\lambda)$ as
\[
\{V_{1},\dots ,V_{\kappa},J_{1}V_{1},\dots ,J_{1}V_{k},\dots ,J_{\nu}V_{1},\dots ,J_{\nu}V_{\kappa}; \; E_{q+j},\;j = 1,\dots ,n-q\},
\]
where $q = (\nu + 1)\kappa,$ the mixed scalar curvature $s_{\rm mix}(\sigma)$ of $\sigma$ is given by
\[
s_{\rm mix}(\sigma) = \sum_{j=1}^{n-q}\sum_{a=1}^{\kappa}\Big (g(R_{V_{a}}E_{q+j},E_{q +j}) + \sum_{s = 1}^{\nu}g(R_{J_{s}V_{a}}E_{q+j},E_{q+j})\Big ) = q(n-q)\lambda.
\]
Now, the result follows from Theorem \ref{tbending}.
\end{proof}

\section{Energy of some special classes of foliations}\indent
A foliation  ${\mathcal F}= {\mathcal F}_{\sigma}$ on a Riemannian manifold $(M,g),$ determined by a completely integrable smooth distribution $\sigma,$ is said to be {\em Riemannian} if it is a {\em transnormal system}, that is, every geodesic that is orthogonal to one leaf remains orthogonal to all the leaves that it intersects \cite{Molino}. Then, the vertical distribution ${\mathcal V}$ is Riemannian. In the subsection $6.3,$ we shall give a family of examples of non Riemannian foliations whose tangent distributions over their regular points are Riemannian. If $(M,g)$ is complete, the transnormal condition implies that the leaves are equidistant to each other.
\subsection{\sc{Tubular and radial foliations.}} We focus on compact (connected) Riemannian manifolds $(M,g)$ equipped with a codimension one Riemannian foliation ${\mathcal F}$ which contains at least one singular leaf $P$ embedded in $M$ $(P$ may be a point). In terms of the first conjugate locus ${\rm Conj}(P)$ along geodesics orthogonal to $P,$ Bolton \cite{Bolton} shows that there are two possibilities for ${\mathcal F}:$
\begin{enumerate}
\item[{\rm Case I}:] ${\rm Conj}(P) = P$ holds, then $P$ is the unique singular leaf, every orthogonal geodesic to $P$ return to {\rm (}possibly a different point of) $P$ in a constant distance $2\mu,$ and $M$ is diffeomorphic to the closed tube $\bar{T}(P,\mu) := T(P,\mu) \cup  P(\mu).$ The regular leaves are tubes around $P,$ or geodesic spheres if $P$ is a point.
\item[{\rm Case II}:] If ${\rm Conj}(P)\neq P$ and $d(P,{\rm Conj}(P)) = \mu$ then $P$ and ${\rm Conj}(P)= P(\mu)$ are the singular leaves and $M$ is diffeomorphic to $\bar{T}(P,\frac{\mu}{2})\cup \bar{T}({\rm Conj}(P),\frac{\mu}{2}),$ or more generally, to $\bar{T}(P,\frac{\mu}{2} + \nu)\cup \bar{T}({\rm Conj}(P),\frac{\mu}{2}-\nu),$ for each $\nu\in ]-\frac{\mu}{2},\frac{\mu}{2}[.$ Each regular leaf is a tube, or a geodesic sphere, around $P$ or around ${\rm Conj}(P).$
\end{enumerate}
We say that ${\mathcal F}$ is a {\em tubular foliation around} $P$ (or a {\em spherical foliation around} a point $x,$ if $P = \{x\})$ and it will be denoted by ${\got T}_{P}$ (or by ${\got E}_{x}.)$ Note that for Case II, one gets ${\mathcal T}_{P} = {\mathcal T}_{{\rm Conj}(P)}.$ Miyaoka \cite{Miyaoka} proves that there exists a {\em transnormal} function $f:M\to \RR$ such that  the connected components of the level sets $f^{-1}(t)$ are precisely the leaves of ${\mathcal T}_{P}.$ Concretely,  on $T(P,\mu),$ $f$ is given by
\begin{equation}\label{f}
f(x)= \cos\frac{\pi}{\mu}r(x),
\end{equation}
where $r$ is the distance function $r = d(P,\cdot).$ In both cases, the tube $T(P,\mu)$ covers $M$ except for the second {\em focal variety} corresponding to the minimum value $-1$ of $f$ and the tubes $P(r),$ $0<r<\mu,$ are all the regular leaves of ${\mathcal T}_{P}.$ 

From  \cite[Lemma 1]{Wang},  the gradient $\nabla f$ of $f$ determines a one-dimensional and totally geodesic almost regular foliation ${\mathcal R}_{P}$ orthogonal to ${\mathcal T}_{P},$ called the {\em radial foliation}  around $P.$ Its singular leaves are then the points of $P$ and of ${\rm Conj}(P).$ From (\ref{f}), one gets
\begin{equation}\label{gradient}
\nabla f = -\frac{\pi}{\mu}(\sin \frac{\pi}{\mu}r)\nabla r
\end{equation}
on $T(P,\mu))\setminus P.$ As a direct consequence of Gauss lemma, $\nabla r$ is the outward radial unit vector field orthogonal to regular leaves of ${\mathcal T}_{P}.$  Denote by $S$ the shape operator (with respect to $\nabla r)$ of the regular leaves $\{P(r)\}_{0<r<\mu}$ of ${\mathcal T}_{P}$ and by $\alpha_{a},$ $a = 1,\dots ,n-1,$ their eigenvalue functions.
\begin{proposition} \label{pone}The tubular and radial foliations ${\mathcal T}_{P}$ and ${\mathcal R}_{P}$ around an embedded submanifold $P$ and the radial vector field $\nabla r$ on $T(P,\mu)\setminus P$ have the same total bending and it is given by
\begin{equation}\label{bendingone}
B({\mathcal T}_{P}) = B({\mathcal R}_{P}) = \frac{1}{2}\sum_{a=1}^{n-1}\int_{0}^{\mu}\Big (\int_{P(r)}\alpha_{a}^{2}\;dv_{P(r)}\Big )dr.
\end{equation}
\end{proposition}
\begin{proof} Taking local orthonormal frames $\{E_{1},\dots ,E_{n-1}; E_{n}= \nabla r\},$ where $E_{1},\dots ,E_{n-1}$ are eigenvectors of $S,$ one directly obtains that $\|\xi\|^{2} = 2\sum_{a=1}^{n-1}\alpha_{a}^{2}$ on $T(P,\mu))\setminus P.$ Hence, using the Fubini's theorem, we obtain (\ref{bendingone}).
\end{proof}
From Proposition \ref{pone} and in accordance with \cite{BGV}, the list of all {\em isoparametric} radial foliations on compact rank one symmetric spaces around points or around totally geodesic submanifolds are given in Table \ref{table1}, where their singular leaves and focal varieties, together with the explicit expressions for their total bendings, are determined. For these Riemann\-ian foliations, the functions $\alpha_{a}$ on each tube $P(r)$ are constant, so they only depend on $r,$ and formula (\ref{bendingone}) reduces to
\[
B({\got R}_{P}) = \frac{1}{2}\sum_{a=1}^{n-1}\int_{0}^{\mu}A^{M}_{P}(r)\alpha_{a}^{2}(r)dr,
\]
where $A^{M}_{P}(r)$ denotes the $(n-1)$-dimensional volume of the tubular hypersurface $P(r).$ All geodesics in ${\mathbb K}P^{m}(\lambda)$ are periodic with the same length ${\it l} = \pi/\sqrt{\lambda}.$ Because radial foliations in Table \ref{table1} are all Case II, we have that $\mu = \pi/2\sqrt{\lambda},$ except for:
\begin{enumerate}
\item[{\rm (i)}] radial foliations around a point $x$ in $\RR P^{m}(\lambda).$ Its focal variety is the regular leaf, known as {\em exceptional}, isometric to $\RR P^{m-1}(\lambda).$ Then, $\mu$ also takes the va\-lues $\pi/2\sqrt{\lambda}.$
\item[{\rm (ii)}] radial foliations around $\RR P^{m}(\lambda)$ (resp., ${\mathbb C}P^{m}(\lambda))$ embedded as totally geo\-de\-sic submanifold of ${\mathbb C}P^{m}(\lambda)$ (resp., ${\mathbb H}P^{m}(\lambda)).$ The geodesics orthogonal to these sub\-ma\-ni\-folds cut them in two points at a distance $\pi/2\sqrt{\lambda}$ and so $\mu = \pi/4\sqrt{\lambda}.$
\end{enumerate}
{\footnotesize
\begin{table}[htb]
\caption{\footnotesize Total bending of isoparametric radial foliations on compact rank one symmetric spaces around points or totally geodesic submanifolds}
\label{table1}
\setlength\arraycolsep{2pt}
\[
\begin{array}{|l|l|l|c|}  \hline
(M,g)                 & \mathrm{Focal}\;\mathrm{varieties}                       & \mathrm{Singular}\;\mathrm{leaves} & B({\mathcal R}_{P})/\mathrm{Vol}(M,g)\\ \hline \hline
S^{m}(\lambda)  & \{x\},\;\{-x\}                              & \{x\},\;\{-x\}                  &    \frac{m-1}{2(m-2)}\lambda \\
(m\geqslant 3) &   S^{m-1}(\lambda),\;\{x,-x\}        & \{x\},\;\{-x\}   &  \frac{m-1}{2(m-2)}\lambda \\
                       & S^{m-p-1}(\lambda), \; S^{q}(\lambda)& S^{m-p-1}(\lambda), \; S^{q}(\lambda)  & \frac{(m-1)[4\delta^{2}-(m-1)^{2} +4]}{2[4\delta^{2}-(m-3)^{2}]}\lambda\\
                       &  (1<p<m-2)                           &                                           &                        \\
                       & S^{m-2}(\lambda),\;S^{1}(\lambda)    & S^{m-2}(\lambda),\;S^{1}(\lambda)   &   \infty \\\hline

\RR P^{m}(\lambda)     & \RR P^{m-1}(\lambda),\;\{x\}                  & \{x\}                    &   \frac{m-1}{2(m-2)}\lambda \\
(m\geqslant 3) & \RR P^{m-p-1}(\lambda), \; \RR P^{p}(\lambda)& \RR P^{m-q-1}(\lambda), \; \RR P^{q}(\lambda) & \frac{(m-1)[4\delta^{2}-(m-1)^{2} +4]}{2[4\delta^{2}-(m-3)^{2}]}\lambda\\
                       &  (1<p<m-2)                           &                                                    &               \\
                       & \RR P^{m-2}(\lambda),\;\RR P^{1}(\lambda)    & \RR P^{m-2}(\lambda),\;\RR P^{1}(\lambda)    &   \infty \\ \hline

{\mathbb C}P^{m}(\lambda)& {\mathbb C}P^{m-1}(\lambda),\;\{x\}        & {\mathbb C}P^{m-1}(\lambda),\;\{x\}          &  \infty \\
 (m\geqslant 2) &   &     &  \\

    & {\mathbb C}P^{m-p-1}(\lambda), \; {\mathbb C}P^{p}(\lambda)& {\mathbb C}P^{m-p-1}(\lambda), \; {\mathbb C}P^{p}(\lambda)& \frac{(m-1)[4\delta^{2}-(m^{2} +1)]}{4\delta^{2}-(m-1)^{2}}\lambda\\
                       &  (1\leq p\leq m-2)                           &                                           &                        \\
                       & \RR P^{m}(\lambda),\;Q                       & \RR P^{m}(\lambda)                    &   \infty \\\hline

{\mathbb H}P^{m}(\lambda)& {\mathbb H}P^{m-1}(\lambda),\; \{x\} & {\mathbb H}P^{m-1}(\lambda),\; \{x\} &  \frac{6m^{2}-5m +2}{2m-1}\lambda\\

                   &{\mathbb H}P^{m-p-1}(\lambda), \; {\mathbb H}P^{p}(\lambda)& {\mathbb H}P^{m-p-1}(\lambda), \; {\mathbb H}P^{p}(\lambda)  &  \frac{8(m-1)\delta^{2}-m(2m^{2} +1)}{4\delta^{2}-m^{2}}\lambda\\
                       &  (1\leq p\leq m-2)                    &                                            &                      \\
                       & {\mathbb C}P^{m}(\lambda),\;\tilde{Q}        & {\mathbb C}P^{m}(\lambda)           &  \frac{10m^{2}-6m-1}{m-1}\lambda\\
                       & (m\geqslant 2)                               &                                           &                        \\ \hline
{\mathbb C}a\;P^{2}(\lambda) & \{x\},\;L & \{x\},\;L               &  \frac{139}{21}\lambda \\\hline

\end{array}
\]
$\delta = ((m-1)/2)-p.$ $Q$ (resp., $\tilde{Q})$ is the tube around $\RR P^{m}(\lambda)$ (resp., ${\mathbb C}P^{m}(\lambda))$ in ${\mathbb C}P^{m}(\lambda)$ (resp., ${\mathbb H}P^{m}(\lambda))$ of radius $\pi/4\sqrt{\lambda}.$
\end{table}}

\begin{theorem}\label{Thsph} The radial foliation ${\got R}_{x}$ {\rm (}resp., the spherical foliation ${\mathcal E}_{x})$ around a point $x$ on $(M,g) = S^{n}(\lambda)$ or $(M,g) = \RR P^{n}(\lambda),$ for $n\geq 3,$ is an absolute minimum for the energy functional on the set of all one-dimensional {\rm (}resp., codimension one{\rm )} almost regular distributions and its total bending is given by
\[
B({\got R}_{x}) = B({\mathcal E}_{x}) = \frac{(n-1)}{2(n-2)}\lambda{\rm Vol}(M,g).
\]
Moreover, for $n\geq 4,$ tangent distributions to radial {\rm (}resp., spherical{\rm )} foliations around points are the only one-dimensional {\rm(}resp., codimension one{\rm )} almost regular distributions to minimize the energy.
\end{theorem}
\begin{proof} On Riemannian manifolds of constant curvature $\lambda,$ the mixed scalar curvature $s_{\rm mix}(\sigma)$ of any $q$-dimensional almost regular distribution $\sigma$ is given by $s_{\rm mix}(\sigma) = q(n-q)\lambda.$ Because the geodesic spheres around points on $S^{n}(\lambda)$ or on $\RR P^{n}(\lambda)$ are totally umbilical hypersurfaces, the first part of the theorem follows directly from Theorem \ref{tbending} (I). For the last part, we use that the foliation whose leaves are round spheres at a constant geodesic distance is the unique codimension one foliation in $S^{n}(\lambda)$ which is Riemannian and totally umbilical. Since $\RR P^{n}(\lambda)$ is obtained from $S^{n}(\lambda)$ by identifying antipodal points, with Riemannian metric such that the two-fold covering map $\tilde{\pi}: S^{n}(\lambda)\to \RR P^{n}(\lambda)$ is a Riemannian submersion, the same result holds for $\RR P^{m}(\lambda).$
\end{proof}

\subsection{\sc{Complex radial foliations.}} Let ${\mathcal T}_{P}$ be a tubular foliation around an embedded submanifold $P$ ($P$ may be a point) in an almost Hermitian manifold $(M,g,J).$ The subset of vector fields $\{ \nabla f, J\nabla f\},$ where $f$ is defined as in (\ref{f}), spans a two-dimensional almost regular invariant distribution called the {\em complex radial} distribution around $P.$ The {\em Hopf vector field} $V = -J\nabla r$ of each tube $P(r),$ $0<r<\mu,$ around $P$ defines a smooth vector field on $T(P,\mu)\setminus P$ belonging to the complex radial distribution. On complex space forms, $V$ is a principal curvature vector field on each $P(r)$ (see \cite{Vanhecke}) and so, $P(r)$ is a nice example of {\em Hopf hypersurface}.

 \begin{proposition}\label{pcomplex} If $(M,g,J)$ is nearly K\"ahler and the regular leaves $P(r),$ $0<r<\mu,$ of ${\mathcal T}_{P}$ are Hopf hypersurfaces, then the complex radial distribution around $P$ is completely integrable and it determines a totally geodesic two-dimensional almost regular invariant foliation with the same singular leaves than those of the radial foliation ${\mathcal R}_{P}.$
\end{proposition}

\begin{proof}  Because $\nabla_{\nabla r} \nabla r = 0,$ it follows from (\ref{f}) and (\ref{gradient}) that
\[
\nabla_{\nabla f}\nabla f = -\frac{\pi}{\mu}\nabla r(\sin \frac{\pi}{\mu}r)\nabla f = -\left(\frac{\pi}{\mu}\right )^{2}f\nabla f.
\]
Hence, using that $(M,g,J)$ is nearly K\"ahler, we have $\nabla_{\nabla f}J\nabla f = -\left(\frac{\pi}{\mu}\right)^{2}fJ\nabla f$ and since $(J\nabla f)(r) = 0$ and $SJV = \alpha JV,$ for some smooth function $\alpha$ on $M,$ one gets
\[
\nabla_{J\nabla f}\nabla f = -\frac{\pi}{\mu}\sin \frac{\pi}{\mu}r\nabla_{J\nabla f}\nabla r = \frac{\pi}{\mu}\sin \frac{\pi}{\mu}rSJ\nabla f = (\frac{\pi}{\mu}\alpha\sin \frac{\pi}{\mu}r)J\nabla f.
\]
So, the complex radial distribution is geodesic. Moreover, we have
\[
[\nabla f,J\nabla f] = -\frac{\pi}{\mu}(\frac{\pi}{\mu}f + \alpha\sin \frac{\pi}{\mu}r)J\nabla f.
\]
Then $\{ \nabla f, J\nabla f\}$ is a {\em locally of finite type} subset of vector fields and, from \cite[Theorem 8.1]{Suss}, the complex radial distribution is completely integrable. Clearly, the points $P\cup{\rm Conj}(P)$ are all the singular leaves. 
\end{proof}

Denote by $\sigma_{P}^{\mathbb C}$ the complex radial distribution and by ${\mathcal R}^{\mathbb C}_{P}$ the corresponding {\em complex radial foliation}. 
\begin{lemma} On ${\mathbb C}P^{m}(\lambda),$ and for each $x\in {\mathbb C}P^{m}(\lambda),$ the orthogonal distribution $(\sigma^{\mathbb C}_{P})^{\bot}_{r}$ of $(\sigma^{\mathbb C}_{P})_{r}$ is umbilical.
\end{lemma}
\begin{proof} Given a unit-speed geodesic $\gamma$ starting at $x = \gamma(0),$ there exists a parallel frame field $\{E_{1},\dots, E_{n}\}$ along $\gamma$ such that $E_{n-1} = J\gamma'$ and $E_{n} = \gamma',$ and where the vectors $E_{i}(r),$ $i = 1,\dots ,n-1,$ are eigenvectors of the shape operator of the small geodesic $S(x,r)$ in ${\mathbb C}P^{m}(\lambda)$ at $\gamma(r).$ The corresponding eigenvalues $\alpha_{i}(r)$ are given by
$\alpha_{1} = \dots = \alpha_{n-2} = \alpha,$ where $\alpha = -\sqrt{\lambda}\cot(\sqrt{\lambda}r),$ and $\alpha_{n-1} = -2\sqrt{\lambda}\cot(2\sqrt{\lambda}r)$ (see \cite{Vanhecke}). Then the non-vanishing components of the intrinsic torsion $\xi$ of $\sigma^{\mathbb C}_{x}$ are 
\[
g(\xi_{E_{j}}E_{j},\gamma'(r)) = g(\xi_{E_{j}}JE_{j},J\gamma'(r)) = -g(\xi_{JE_{j}}E_{j},J\gamma'(r)) = \alpha,
\]
for $j = 1,\dots,n-2.$ It proves that $(\sigma^{\mathbb C}_{P})^{\bot}_{r}$ is umbilical.
\end{proof}
From  Corollary \ref{pinv} (I)(i), we have the following result.
\begin{theorem} The complex radial foliation ${\mathcal R}_{x}^{\mathbb C}= {\mathcal R}_{{\mathbb C}P^{1}(\lambda)}^{\mathbb C}$ in the complex projective plane ${\mathbb C}P^{2}(\lambda)$ is an absolute minimum for the energy functional on the set of all two-dimensional invariant almost regular distributions and
\[
B({\mathcal R}^{\mathbb C}_{x}) = 2\lambda {\rm Vol}({\mathbb C}P^{2}(\lambda)).
\]
\end{theorem}

\subsection{\sc{Deformations of a tubular foliation.}} Let ${\mathcal T}_{P}$ be the tubular foliation around an embedded submanifold $P$ of a compact Riemannian manifold $(M,g)$ and denote by $\sigma$ its tangent distribution. Let $f:M\to \RR$ be the associated transnormal function given in (\ref{f}). For each $\varepsilon\in [0,\frac{\pi}{2}],$ we construct a new distribution $\sigma_{\varepsilon}$ given by
$$
\sigma_{\varepsilon}(x) =
\left\{
\begin{array}{lcl}
\sigma(x),& & \mbox{\rm if}\;|f(x)|\leq \sin\varepsilon\;\mbox{\rm or}\;f(x) = \pm 1;\\[0.4pc]
T_{x}M,& & \mbox{\rm if}\; \sin\varepsilon < |f(x)| < 1.
\end{array}
\right.
$$
Note that $\sigma_{0}$ is a trivial distribution and $\sigma_{\pi/2}$ coincides with $\sigma.$ 

\begin{lemma} The distribution $\sigma_{\varepsilon},$ for each $\varepsilon \in [0,\frac{\pi}{2}],$ is smooth and completely inte\-gra\-ble.
\end{lemma}
\begin{proof} Let $D$ be any subset of ${\mathfrak X}_{loc}(M)$ spanning $\sigma.$ Then $D\cup\{(\alpha\comp f)\nabla f\}$ determines $\sigma_{\varepsilon},$ where $\alpha$ is a smooth real function  which is $0$ on $[-\sin\varepsilon,  \sin\varepsilon ]$ and positive out this interval, e.g.
$$
\alpha(t) =
\left\{
\begin{array}{lcl}
e^{-\frac{1}{(t + \sin\varepsilon)^{2}}}, & & \mbox{\rm if}\;t<-\sin\varepsilon;\\[0.4pc]
0, & &\mbox{\rm if}\; -\sin\varepsilon \leq t \leq  \sin\varepsilon;\\[0.4pc]
e^{-\frac{1}{(t-\sin\varepsilon)^{2}}}, & & \mbox{\rm if}\;t> \sin\varepsilon.
\end{array}
\right.
$$
Then, $\sigma_{\varepsilon}$ is smooth and, because $D\cup\{(\alpha\comp f)\nabla f\}$ can be taken {\em locally of finite type}, it follows from \cite[Theorem 8.1]{Suss} that $\sigma_{\varepsilon}$ is completely integrable. 
\end{proof}Then $\sigma_{\varepsilon}$ determines a foliation $({\mathcal T}_{P})_{\varepsilon,}$ called the $\varepsilon$-{\em deformation} of ${\mathcal T}_{P},$ whose singular leaves are the singular leaves of ${\mathcal T}_{P},$ that is, the submanifold $P$ in Case I and $P$ and ${\rm Conj}(P)$ in Case II, together with the level sets $f^{-1}(\sin\varepsilon)$ and $f^{-1}(-\sin\varepsilon).$ 
\begin{remark}{\rm Clearly $({\mathcal T}_{P})_{\varepsilon},$ for all $\varepsilon \in ]0,\pi/2 [,$ is not almost regular and it is not a Riemannian foliation. Nevertheless, its tangent distribution $\sigma_{\varepsilon}$ is Riemannian.}
\end{remark}
\begin{proposition} We have:
\begin{equation}\label{int3}
B(({\mathcal T}_{P})_{\varepsilon}) = \frac{1}{2}\sum_{a=1}^{n-1}\int_{\frac{\mu}{2\pi}(\pi-2\varepsilon)}^{\frac{\mu}{2\pi}(\pi+2\varepsilon)}\Big (\int_{P(r)}\alpha_{a}^{2}\;dv_{P(r)}\Big )dr,
\end{equation}
where $\alpha_{a},$ $a=1,\dots n-1,$ are the eigevalues functions for the shape operator of the level sets of ${\mathcal T}_{P}.$
\end{proposition}
\begin{proof} The set of the regular points $M_{r}(\varepsilon)$ of $\sigma_{\varepsilon}$ is the union of the open subset
$$
\begin{array}{lcl}
M_{r}(\varepsilon) & = & M_{r}^{1}(\varepsilon) \cup M_{r}^{2}(\varepsilon) = \{x\in M_{r}(\varepsilon)\mid d(x) = n-1\}\cup \{x\in M_{r}(\varepsilon)\mid d(x) = n\}\\[0.4pc]
 & = &  f^{-1}(]-\sin\varepsilon, \sin\varepsilon [)\cup f^{-1}(]-1,-\sin\varepsilon [ \cup ]\sin\varepsilon, 1[).
\end{array}
$$
Moreover,  $M_{r}^{1}(\varepsilon)$ can be expressed as
\[
M_{r}^{1}(\varepsilon) = \{\exp_{P}ru\;\mid \; u\in T^{\bot}P,\; \|u\| = 1,\;\; |r-\frac{\mu}{2}|< \frac{\mu\varepsilon}{\pi}\}.
\]
Since the intrinsic torsion of $\sigma_{\varepsilon}$ vanishes on $M_{r}^{2}(\varepsilon)$ and it coincides with the intrinsic torsion $\xi^{1}$ of $\sigma$ on $M_{r}^{1}(\varepsilon),$ the result follows directly using (\ref{bendingone}).
\end{proof}

\begin{example}{\rm Let ${\mathcal E}_{x}$ be the spherical foliation around a point $x$ in the $2$-sphere $S^{2}(\lambda).$ Then a transnormal function $f,$ such that ${\mathcal E}_{x} = {\mathcal F}_{f},$ is defined as $f(\exp_{x}rv) = \cos t\sqrt{\lambda}r,$ where $v\in T_{x}S^{n}(\lambda)$ and $\|v\| = 1.$ Its level sets are round circles centered at $x$ with $\alpha(r)=-\sqrt{\lambda}\cot t\sqrt{\lambda}r$ and $A_{x}^{S^{2}(\lambda)}(r)= \frac{2\pi}{\sqrt{\lambda}}\sin t\sqrt{\lambda}r.$ From Theorem \ref{Tori}, we know that $B({\mathcal E}_{x})= \infty.$ Nevertheless, the $\varepsilon$-deformation $({\mathcal E}_{x})_{\varepsilon}$ of ${\mathcal E}_{x},$ for all $\varepsilon\in [0,\frac{\pi}{2}[,$ has finite total bending. In fact, from (\ref{int3}), we get
\[
B({\mathcal F}_{\varepsilon}) = \pi\int_{\pi/2 -\varepsilon}^{\pi/2 + \varepsilon} \cos^{2}t\sin^{-1}t\;dt = \pi\Big( \ln\left(\frac{1+ \sin\varepsilon}{1-\sin\varepsilon}\right) -2\sin\varepsilon\Big).
\]
}
\end{example}

\end{document}